\documentclass[letterpaper, 10 pt, journal, twoside]{IEEEtran}
\IEEEoverridecommandlockouts

\def\BibTeX{{\rm B\kern-.05em{\sc i\kern-.025em b}\kern-.08em
    T\kern-.1667em\lower.7ex\hbox{E}\kern-.125emX}
}%

\newcommand\copyrighttext{%
  \footnotesize \textcopyright 2021 IEEE. Personal use of this material is permitted.
  Permission from IEEE must be obtained for all other uses, in any current or future
  media, including reprinting/republishing this material for advertising or promotional
  purposes, creating new collective works, for resale or redistribution to servers or
  lists, or reuse of any copyrighted component of this work in other works.
  DOI: \href{https://ieeexplore.ieee.org/document/9439931}{10.1109/LCSYS.2021.3083467}}
\newcommand\copyrightnotice{%
\begin{tikzpicture}[remember picture,overlay]
\node[anchor=south,yshift=5pt] at (current page.south) {\fbox{\parbox{\dimexpr\textwidth-\fboxsep-\fboxrule\relax}{\copyrighttext}}};
\end{tikzpicture}%
}

    \usepackage{cite}
    \usepackage{amsmath,amssymb,amsfonts, amsthm}
    \usepackage{algorithmic}
    \usepackage{mathtools}
    \usepackage{graphicx}
    \usepackage{textcomp}
    \usepackage{xcolor}
    \usepackage{hyperref}

    \usepackage{graphicx}
    \graphicspath{{figures/approach/}{figures/experiments/}{figures/experiments/MSP/}{figures/experiments/IVP/}}

    \newcommand{\secref}[1]{Section~\ref{#1}}
    \newcommand{\lemref}[1]{Lemma~\ref{#1}}
    
    \newcommand{\theoref}[1]{Theorem~\ref{#1}}
    \renewcommand{\eqref}[1]{(\ref{#1})}
    \newcommand{\figref}[1]{Figure~\ref{#1}}
    \newcommand{\tabref}[1]{Table~\ref{#1}}
    \newcommand{\lineref}[1]{Line~\ref{#1}}

    \DeclareMathOperator*{\minimize}{\mathrm{minimize}}

    \usepackage{bm}

    \usepackage{tikz}
    \usetikzlibrary{shapes.geometric}
    \usetikzlibrary{arrows.meta,arrows}
    \usetikzlibrary{plotmarks}
    \usetikzlibrary{decorations.pathmorphing,patterns}
    \tikzset{>=latex}

    \usepackage{pgfplots}
    \pgfplotsset{compat=1.16}
    \newlength\fheight
    \newlength\fwidth
    \renewcommand{\theenumi}{(\roman{enumi})}%

    \usepackage{booktabs}
    \newcommand{\ra}[1]{\renewcommand{\arraystretch}{#1}}

	\usepackage[nocomma]{optidef}

    \usepackage{multirow}

    \newcommand{\R}{\mathbb{R}}

    \newcommand{\I}{\mathbb{I}}

    \newcommand{\sign}{\mathrm{sgn}}

    \newcommand{\penref}{\rho\mathrm{LCQP}}
    \newcommand{\lcqpref}{\mathrm{LCQP}}

    \newcommand{\xkj}{x_{kj}}
    \newcommand{\dkj}{d_{kj}}
    \newcommand{\pkj}{p_{kj}}
    \newcommand{\alphakj}{\alpha_{kj}}
    \newcommand{\qkj}{q_{kj}}
    \newcommand{\lkj}{\ell_{kj}}
    \newcommand{\thetakj}{\vartheta_{kj}}
    \newcommand{\gammakj}{\gamma_{kj}}
    \newcommand{\deltakj}{\delta_{kj}}

    \usepackage{upgreek}
    \renewcommand{\tau}{\uptau}

    \usepackage[ruled,vlined,linesnumbered]{algorithm2e}
    \newtheorem{theorem}{Theorem}
    \newtheorem{lemma}{Lemma}
    \newtheorem{definition}{Definition}
    \newtheorem{remark}{Remark}
    \newcommand{\expnumber}[2]{{#1}\mathrm{e}{#2}}

    \newcommand{\Nurkanovic}{Nurkanovi\'{c}}

    \usepackage{bigdelim}

    \newcommand{\KKT}{\textup{KKT}}
    \newcommand{\LICQ}{\textup{LICQ}}

    \newcommand{\LCQP}{\textup{LCQP}}

\begin{document}

\title{A Sequential Convex Programming Approach to Solving Quadratic Programs and Optimal Control Problems with Linear Complementarity Constraints}

\author{\IEEEauthorblockN{}
\IEEEauthorblockA{\textit{Systems Control and Optimization Laboratory} \\
\textit{Albert Ludwig University of Freiburg}\\
Freiburg, Germany
}
}

\author{Jonas Hall$^{1}$, Armin \Nurkanovic$^{1,2}$, Florian Messerer$^{1}$, Moritz Diehl$^{1,3}$%
\thanks{This research was supported by the German Federal Ministry for Economic Affairs and Energy (BMWi) via DyConPV (0324166B), by DFG via Research Unit FOR 2401 and project 424107692, and by the German Federal Ministry of Education and Research (BMBF) via the funded Kopernikus project: SynErgie (03SFK3U0)}%
\thanks{$^{1}$Department of Microsystems Engineering (IMTEK), University of Freiburg, 79110 Freiburg, Germany {\tt\small \{jonas.hall, florian.messerer, moritz.diehl\}@imtek.de}}%
\thanks{$^{2}$Siemens Technology, 81739 Munich, Germany
{\tt\small {armin.nurkanovic}@siemens.com}}%
\thanks{$^{3}$Department of Mathematics, University of Freiburg, 79110 Freiburg, Germany}%
}%

\clearpage\maketitle
\thispagestyle{empty}
\copyrightnotice

\begin{abstract}
    Mathematical programs with complementarity constraints are notoriously difficult to solve due to their nonconvexity and lack of constraint qualifications in every feasible point. This work focuses on the subclass of quadratic programs with linear complementarity constraints. A novel approach to solving a penalty reformulation using sequential convex programming and a homotopy on the penalty parameter is introduced. Linearizing the necessarily nonconvex penalty function yields convex quadratic subproblems, which have a constant Hessian matrix throughout all iterates. This allows solution computation with a single KKT matrix factorization. Furthermore, a globalization scheme is introduced in which the underlying merit function is minimized analytically, and guarantee of descent is provided at each iterate. The algorithmic features and possible computational speedups are illustrated in a numerical experiment.
\end{abstract}

\section{Introduction}
Linear Complementarity Quadratic Programs (LCQP) are quadratic programs with additional complementarity constraints. The complementarity conditions consist of inequality constraints, imposing nonnegativity of the complementary pairs, and a bi-linear equality constraint imposing orthogonality. In order to formalize this, consider an $n$-dimensional input space with $n_C$ complementarity constraints. Let $L,R \in \R^{n_C \times n}$ be the linear input transformations selecting the complementarity pairs. Then a general LCQP can be written as
\begin{mini!}
	{x \in \R^n}{\frac{1}{2}x^\top Q x + g^\top x}
	{\label{eq:LCQP}}{\lcqpref:\quad}
	\addConstraint{0}{\leq A x - b, \label{eq:LCQP:A}}
	\addConstraint{0}{\leq Lx \perp Rx \geq 0, \label{eq:LCQP:comp}}
\end{mini!}
where $0 \prec Q \in \R^{n \times n}, g \in \R^n, A \in \R^{n_A \times n}$, and $b \in \R^{n_A}$. The complementarity constraint~\eqref{eq:LCQP:comp} is a compact notation for
\begin{equation}\label{eq:comp:detail}
    0 \leq Lx \perp Rx \geq 0 ~\Longleftrightarrow~
        \begin{cases}
            0 \leq Lx \\
            0 \leq Rx \\
            0 = x^\top L^\top R x.
        \end{cases}
\end{equation}
An illustrative example is depicted in~\figref{fig:warm:up}. These problems are particularly difficult to solve due to their nonconvexity and nonsmoothness of the feasible set. Moreover, standard constraint qualifications such as the the Linear Independence Constraint Qualification (LICQ) or the weaker Mangasarian-Fromovitz constraint qualification are violated at every feasible point~\cite[Proposition 1.1]{Ye1997}. Thus, it is very difficulty to numerically solve~\eqref{algo:LCQP} directly, as the multipliers are unbounded and the constraint Jacobian matrices are degenerate. Generalizations of~\eqref{eq:LCQP} with nonlinear functions are known as Mathematical Programs with Complementarity Constraints (MPCC). These problems have received a lot of attention and many solution strategies have been proposed, many of which are included in the survey~\cite{Kim2020MPEC}.

\begin{figure}
    \centering
    \resizebox{0.7\linewidth}{!}{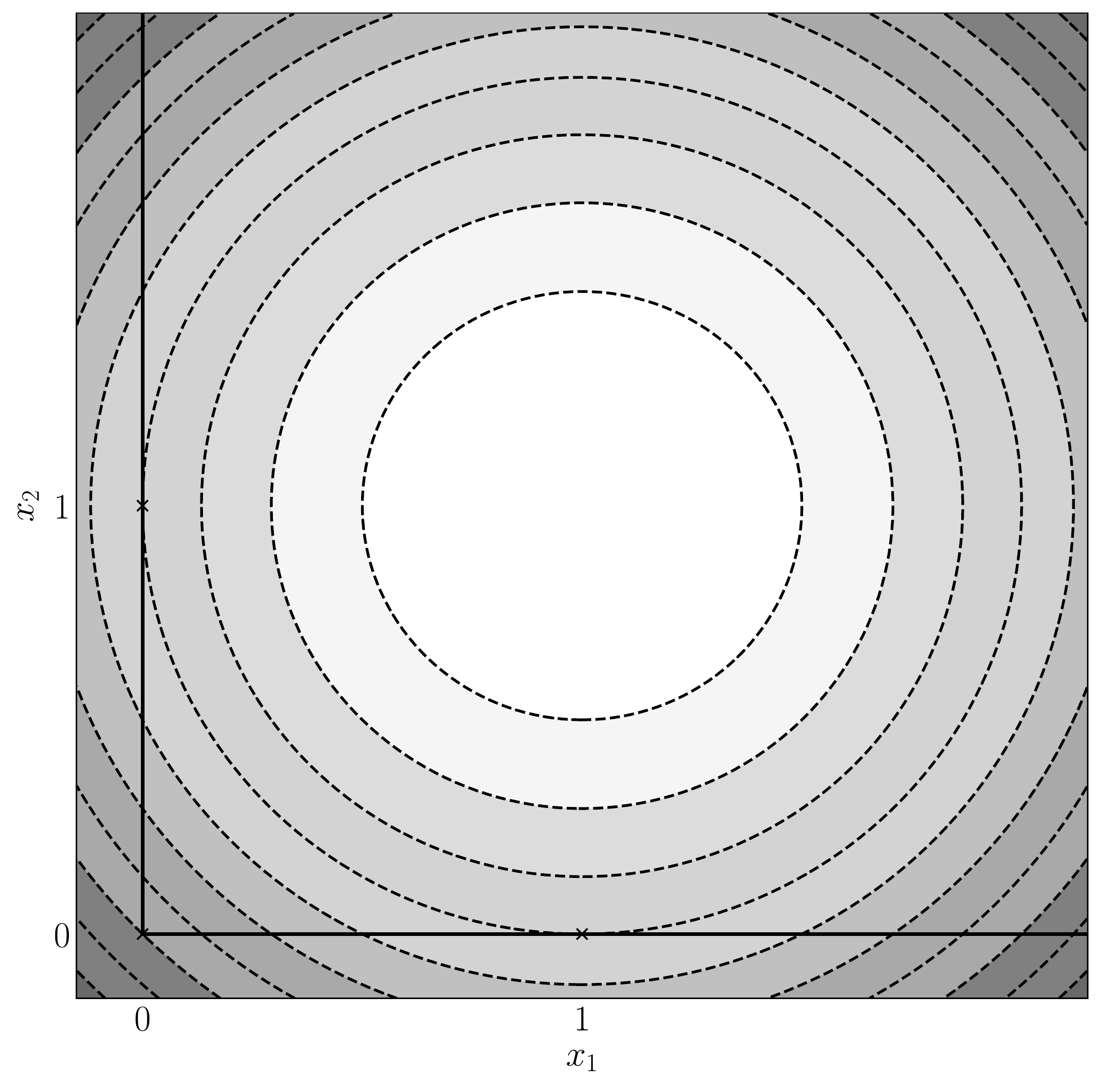}
    \caption{An illustrative LCQP with $Q = 2\I_2$, $g = (-2, -2)^\top$, $L = (1,0)$, and $R = (0,1)$, as originally presented in~\cite{scheel2000mathematical}. The feasible set is depicted by the solid black line. This example contains two strongly stationary points located at $(0, 1)$ and $(1, 0)$ (which are local minima), and one weaker \mbox{(Clarke-)stationary} point at the origin (which is a local maximum).}
    \label{fig:warm:up}
\end{figure}

One popular approach to reformulate an MPCC into a less degenerate Non-Linear Program (NLP) is to remove the bi-linear term from the constraint and penalize its violation in the objective. Convergence and solution equivalence of the two approaches have been studied for example in~\cite{ralph2004some}. This approach is further adapted to an interior point strategy in~\cite{leyffer2006interior}. Similar penalty reformulations can be found in~\cite{abdallah2019solving}, and a more generic form in~\cite{anitescu2005}. Instead of penalizing the bi-linear term, authors have suggested to replace~\eqref{eq:LCQP:comp} with nonlinear complementarity functions, which vanish exactly on the complementarity set. A popular example is the Fisher-Burmeister function~\cite{fischer1995ncp} and adaptations of it~\cite{chen2000penalized}. MPCCs have also been solved as nonlinear programs by replacing the complementarity product with an inequality constraint~\cite{fletcher2004solving}, or with \mbox{(in)equality} regularization schemes~\cite{ralph2004some}. All these approaches lead to solving NLPs where constraint qualifications are satisfied. Usually a sequence of relaxed problems must be solved to recover a solution of the initial MPCC~\cite{ralph2004some, leyffer2006interior}. Their limits in applications to direct optimal control have been shown in~\cite{nurkanovic2020limits}.

MPCCs appear in a wide range of applications in engineering, economics and science. An extensive list of applications is presented in~\cite{ferris1997engineering}. Typical applications in mechanics appear in the context of friction problems or with impacts of rigid bodies~\cite{stewart2011dynamics}. Many of the MPCC examples appear as LCQPs, e.g. in~\cite{stewart1996numerical} or~\cite{stewart2010optimal}.

Other authors have focused on this subclass before. An analysis of C-stationary points using homotopy approaches is provided in~\cite{ralph2011c}. They also appear in the form of subproblems, when nonlinear programs with linear complementarity constraints are solved with a sequential quadratic programming approach~\cite{fukushima1998globally}. LCQPs are also addressed as Mixed-Integer Quadratic Programs (MIQP), for example in a combination with the Benders scheme in~\cite{bai2013convex} or using branch-and-bound techniques~\cite{stellato2018embedded}. Similar to the approach presented here, the latter exploits linear algebra structures by reusing a factorized matrix. MIQP methods can be advantageous in some cases, in particular due to their ability of finding global solutions. However, their relaxations are weak which might result in large search trees and they are thus limited to formulations without too many integer variables~\cite{pang2010three}.

The contribution of this paper is the introduction of a novel algorithm for solving LCQPs. The solution strategy is based on an existing penalty reformulation, which is solved with a Sequential Convex Programming (SCP) approach. Each iterate within the convex programming loop is shown to reduce the merit function. The according step length is controlled by an analytical globalization scheme, which minimizes the merit function in every step.

This paper is structured as follows: \secref{sec:preliminaries} provides the required concepts, including the exactness of the underlying penalty approach. In \secref{sec:algorithmic:development}, the novel approach to solving LCQPs using an SCP technique is presented. Further, a globalization scheme is demonstrated by analytically minimizing an exact merit function. A comparison to three state-of-the-art solution variants is provided in~\secref{sec:numerical:example} by solving an illustrative Optimal Control Problem (OCP). \secref{sec:conclusion:future:work} concludes the paper and highlights further algorithmic improvements.

\section{Penalty Reformulation}\label{sec:preliminaries}
This section briefly addresses the theory of stationarity for MPCCs and the underlying penalty reformulation with its convergence properties. Consider the general LCQP~\eqref{eq:LCQP}. Denote by $\mathcal{L}(x)$ and $\mathcal{R}(x)$ the sets of active constraints among those of $Lx \geq 0$ and $Rx \geq 0$, respectively. Further, let $\mathcal{W}(x) = \mathcal{L}(x) \cap \mathcal{R}(x)$ denote the set of weakly active complementarity pairs. In contrast, let $\bar{\mathcal{L}} = \mathcal{L}\setminus \mathcal{W}$ and $\bar{\mathcal{R}} = \mathcal{R} \setminus \mathcal{W}$ refer to the strongly active complementarity pairs, respectively.
\begin{definition}
    A feasible point $x$ of \LCQP~\eqref{eq:LCQP} is called strongly stationary, if there exist dual variables $y = (y_A, y_L, y_R) \in \R^{n_A} \times \R^{n_C} \times \R^{n_C}$ satisfying
\begin{subequations}\label{eq:weak:stationarity}
	\begin{align}
		Qx + g - A^\top y_A - L^\top y_L - R^\top y_R &= 0, \\
		\min(Ax - b, y_A) &= 0, \\
		y_{L_i} &= 0, \quad i \in \bar{\mathcal{R}}(x), \\
		y_{R_i} &= 0, \quad i \in \bar{\mathcal{L}}(x), \\
        y_{L_i}, y_{R_i} &\geq 0, \quad i \in \mathcal{W}(x).
	\end{align}
\end{subequations}
\end{definition}
\noindent For more details on stationarity of MPCCs, including other stationarity types, refer to~\cite[Section 2]{guo2015solving}.

A popular approach to solving LCQPs are penalty reformulations, and the here presented algorithm is based on a technique as discussed for example in~\cite[Section 1]{ralph2004some}. Consider the penalty function
\begin{equation}\label{eq:pen:fun}
    \varphi(x) = x^\top L^\top R x = \frac{1}{2} x^\top (L^\top R + R^\top L) x = \frac{1}{2} x^\top C x,
\end{equation}
where $C \in \R^{n \times n}$ is the symmetrization of the complementarity product.
\begin{remark}\label{rem:C:indefinite}
    The matrix $C$ is usually indefinite. This is necessarily the case if $L$ and $R$ consist of pairwise orthogonal rows, e.g. if each row selects a distinct optimization variable. For example, the curvature of the complementarity product $x_1^\top x_2$ at the origin towards $(1, 1)$ is positive, whereas it is negative towards $(-1, 1)$.
\end{remark}
The penalty reformulation is obtained by replacing the bi-linear complementarity constraint from~\eqref{eq:LCQP:comp} with the penalty function \eqref{eq:pen:fun} in the objective. The resulting QP reads as
\begin{mini!}
	{x \in \R^n}{\frac{1}{2}x^\top Q x + g^\top x + \rho \cdot \varphi(x) \label{eq:LCQP:pen:outer:obj}}
	{\label{eq:LCQP:pen:outer}}{\penref:\quad}
	\addConstraint{0}{\leq A x - b, \label{eq:LCQP:pen:outer:A}}
    \addConstraint{0}{\leq Lx, \label{eq:LCQP:pen:outer:compl}}
    \addConstraint{0}{\leq Rx, \label{eq:LCQP:pen:outer:comp2}}
\end{mini!}
where $\rho > 0$ is the respective penalty parameter. Throughout this paper all solutions of $\penref$ are assumed to satisfy LICQ. Under this assumption a convergence property of the penalized approach is captured in the following theorem, which represents a special case of the theorem proven by Ralph and Wright in~\cite[Section 5]{ralph2004some}.
\begin{theorem}\label{theo:penalty:convergence}
    Let $\penref$~\eqref{eq:LCQP:pen:outer} satisfy {\LICQ} at $x^\ast \in \R^n$, then the following statements hold:
        \begin{enumerate}
            \item If $x^\ast$ is a strongly stationary point of the \LCQP~\eqref{eq:LCQP}, then there exists a finite $\widetilde{\rho}$, such that $x^\ast$ is a {\KKT} point of~\eqref{eq:LCQP:pen:outer} for all $\rho > \widetilde{\rho}$.
            \item If $x^\ast$ is a {\KKT} point of~\eqref{eq:LCQP:pen:outer} and $\varphi(x^\ast) = 0$, then $x^\ast$ is a strongly stationary point of the \LCQP.
        \end{enumerate}
\end{theorem}

\section{Algorithmic Development}\label{sec:algorithmic:development}
The above motivates finding stationary points of the LCQP by solving~\eqref{eq:LCQP:pen:outer} for a penalty large enough to ensure satisfaction of the complementarity constraints. In this section an approach to finding solutions via a penalty homotopy together with a sequential convex programming approach is introduced. Finally, an analytical globalization scheme with a guarantee of descent is provided.

\subsection{Penalty Homotopy}
For a given penalty parameter $\rho_k > 0$ the respective penalty reformulation~\eqref{eq:LCQP:pen:outer} is solved. Subsequently, the penalty parameter $\rho_{k+1} = \beta \rho_k$ is updated for a fixed $\beta > 1$. This procedure is repeated until complementarity is satisfied. \theoref{theo:penalty:convergence} ensures that if the LCQP has a strongly stationary point, then this point must also be a KKT point of the penalized reformulation with respect to a finite penalty parameter $\tilde{\rho}$. One could consider simply solving~\eqref{eq:LCQP:pen:outer} for a very large penalty in the hope of instantly satisfying complementarity. However, these penalty reformulations often become ill-conditioned for large penalty parameters as the nonconvex part becomes dominant. Due to this nonconvexity, the original LCQP may contain many local solutions. Solving the penalized subproblem for a very small penalty parameter leads to a solution close to the global minimum of \eqref{eq:LCQP} without the bi-linear complementarity constraint, whereas a solution with respect to a large penalty parameter favors complementarity satisfaction. This motivates a homotopy on the penalty parameter with the aim of finding a good local solution~\cite{leyffer2006interior}. Yet, this is only a heuristic and there is no guarantee of finding the global minimizer as the original NLP~\eqref{algo:LCQP} is nonconvex. Further, a homotopy often avoids convergence to spurious solutions, as for example shown for OCPs with discontinuous systems~\cite{nurkanovic2020limits}. The sequence of these penalized subproblems~\eqref{eq:LCQP:pen:outer} is denoted as the outer loop.

\subsection{The Sequential Convex Programming Approach}
Each outer loop problem is solved using sequential convex programming, resulting in an inner loop. Let $k$ and $j$ denote the outer and inner loop indices, respectively. The very first inner loop is initialized with the initial guess, whereas all consecutive inner loops are initialized with the previous iterate. The penalty function is approximated at each iterate $\xkj$ using its first-order Taylor expansion
\begin{subequations}
    \begin{align*}
        \varphi(x) &\approx \varphi(\xkj) + (x - \xkj)^\top \nabla \varphi(\xkj) \\
        &= \left( \varphi(\xkj) - \xkj^\top C \xkj \right) + x^\top C \xkj.
    \end{align*}
\end{subequations}
Now let $d_{kj} = C \xkj$, and note that $x^\top d_{kj}$ is the only term dependent on $x$. The penalty function is replaced by $d_{kj}^\top x$ resulting in the convex inner loop subproblem
\begin{mini!}
	{x \in \R^n}{\frac{1}{2}x^\top Q x + \left(g + \rho_k d_{kj}\right)^\top x  \label{eq:LCQP:pen:inner:obj}}
	{\label{eq:LCQP:pen:inner}}{}
	\addConstraint{0}{\leq A x - b, \label{eq:LCQP:pen:inner:A}}
    \addConstraint{0}{\leq Lx, \label{eq:LCQP:pen:inner:compl}}
    \addConstraint{0}{\leq Rx. \label{eq:LCQP:pen:inner:comp2}}
\end{mini!}
Denote the unique minimizer of the inner loop subproblem by $\xkj^\ast$ and the according step by $\pkj = \xkj^\ast - \xkj$. Given this inner solution an optimal step length $\alphakj$ is obtained from the globalization scheme described in \secref{subsec:globalization}. Finally, the step update $x_{k,j+1} = \xkj + \alphakj \pkj$ is performed. The inner loop is terminated once a KKT point of the respective outer loop problem \eqref{eq:LCQP:pen:outer} is found. The following lemma relates the minimizers of the inner loop problems with the KKT points of the outer loop. A proof for a more general case can be found in~\cite[Lemma 4.1]{messerer2021survey}.
\begin{lemma}\label{lem:inner:outer:stationarity}
    Let $(\xkj, \rho_k)$ be a feasible iterate of~\eqref{eq:LCQP:pen:inner}. Then the respective inner loop minimizer $\xkj^\ast$ agrees with $\xkj$ iff $\xkj$ is a {\KKT} point of the outer loop problem~\eqref{eq:LCQP:pen:outer} with respect to $\rho_k$.
\end{lemma}

\begin{figure}
    \centering
    \resizebox{\linewidth}{!}{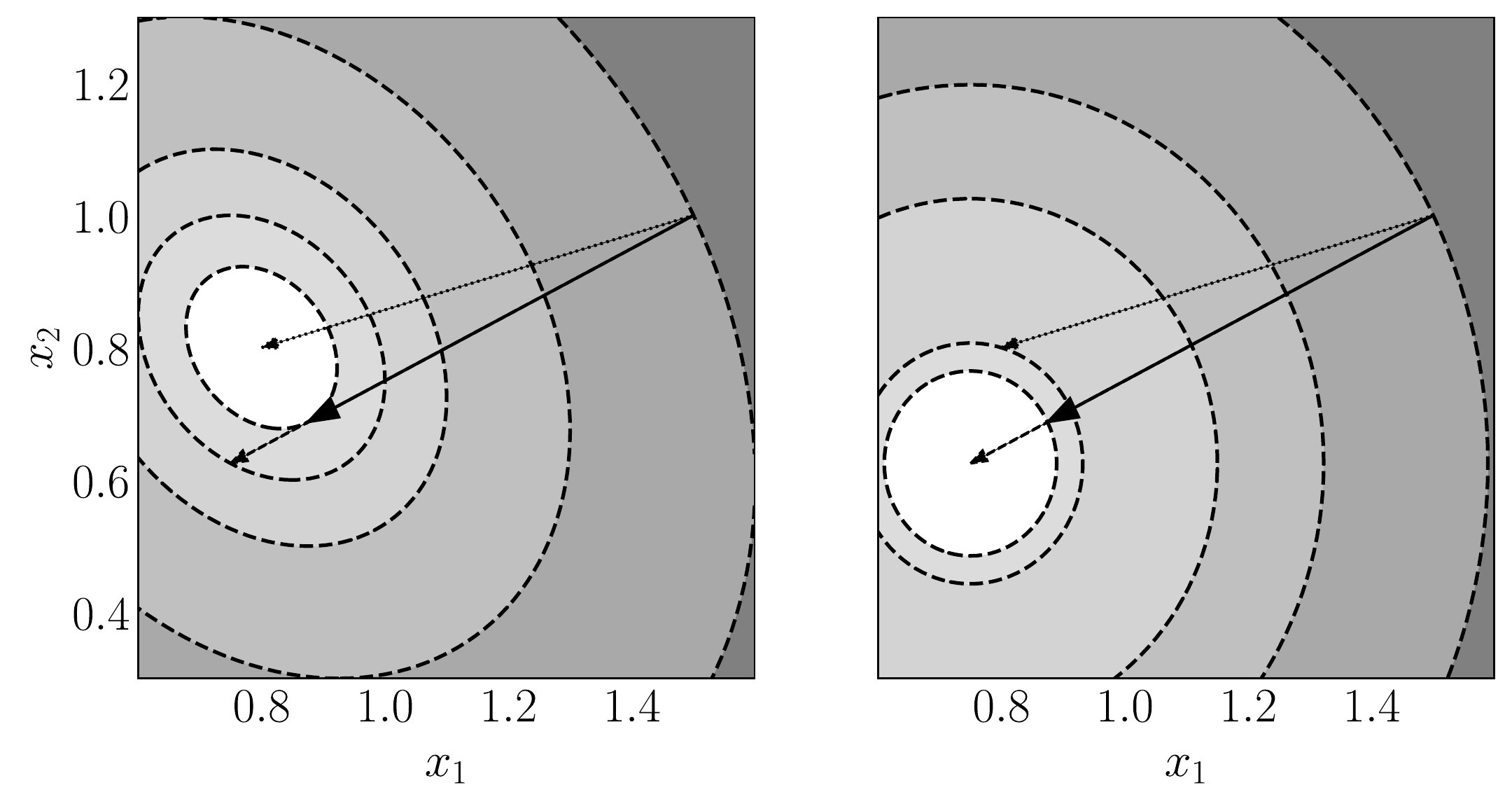}
    \caption{The left plot shows the level lines of the NLP objective~\eqref{eq:LCQP:pen:outer:obj}, and the right plot depicts the level lines of the convex QP subproblem objective~\eqref{eq:LCQP:pen:inner:obj}. Both plots include a total of three vectors. Out of the parallel vectors, the dashed one shows $\pkj = \xkj^\ast - \xkj$, and the solid one is $\xkj + \alphakj \pkj$. The dotted vector indicates the stationary point of the outer loop problem~\eqref{eq:LCQP:pen:outer}. Note that the chosen step length minimizes~\eqref{eq:LCQP:pen:outer:obj} on the path $\xkj + \alpha \pkj$.}
    \label{fig:globalization:step}
\end{figure}

The algorithm is terminated once a complementarity satisfying KKT point of~\eqref{eq:LCQP:pen:outer} is found. Theorem~\ref{theo:penalty:convergence} indicates that the solution must be a strongly stationary point of the original LCQP, under the assumption of exact complementarity satisfaction.

There are two reasons why it can be attractive to replace the full penalty function by its linear approximation. First, convex subproblems are obtained at the cost of the additional inner loop. While the original formulation becomes more and more indefinite as the penalty parameter grows, convexity of the inner loop subproblem is always ensured, as the Hessian matrix remains to be $Q$ in every subproblem. This also induces the second advantage: the Hessian and constraint matrices remain constant over all (inner and outer) iterates. Consequently, the KKT matrix factorization can be reused, and each inner loop problem can be solved efficiently, e.g. by making use of the hot-starting technique employed in active-set solvers such as qpOASES~\cite{Ferreau2008}. With the computation of factorizations being a significant expense, this advantage can outweigh the cost of inner loop iterations, as demonstrated in \secref{sec:numerical:example}.

\subsection{Optimal Step Length Globalization}\label{subsec:globalization}
Consider the merit function
\begin{equation}
    \psi(x, \rho) = \frac{1}{2} x^\top (Q + \rho C) x + g^\top x,
\end{equation}
which coincides with the outer loop objective function. On the other hand, the inner loop objective function provides the strictly convex quadratic model
\begin{equation}\label{eq:quadratic:model}
    \thetakj(x) = \frac{1}{2} x^\top Q x + (g + \rho_k \dkj)^\top x.
\end{equation}
An analytical globalization scheme by is introduced by solving
\begin{equation}\label{min:merit:function}
    \minimize_{\alpha \in [0,1]}\quad \psi(\xkj + \alpha \xkj^\ast, \rho).
\end{equation}
This concept is visualized in \figref{fig:globalization:step}. Evaluating the objective function in~\eqref{min:merit:function} yields the quadratic polynomial in $\alpha$
\begin{equation}\label{eq:line:search:objective}
    \psi(\xkj + \alpha \pkj, \rho_k) = \frac{1}{2}\alpha^2 \qkj + \alpha \lkj + \psi(\xkj, \rho_k),
\end{equation}
where
\begin{subequations}\label{eq:LCQP:OptimalStepLength:Components}
    \begin{align}
        \gammakj &= \pkj^\top Q  \pkj, \\
        \deltakj &= \pkj^\top  \rho_k C \pkj, \\
        \qkj &= \gammakj + \deltakj, \\
        \lkj &= \xkj^\top \left(Q + \rho_k C \right) \pkj + g^\top \pkj.
    \end{align}
\end{subequations}

There are two different cases to be considered for solving~\eqref{min:merit:function}, both related to the sign of $\deltakj$. Both cases are handled individually and their geometric meanings are discussed.

First consider the sign of the linear component. Observe that $\lkj$ represents the directional derivative of the merit function along $\pkj$, i.e.
\begin{equation}
    \nabla \psi(\xkj, \rho_k)^\top \pkj = ((Q + \rho_k C) \xkj + g)^\top \pkj = \lkj.
\end{equation}
This provides descent along $\pkj$ at $\psi(\xkj, \rho_k)$ given $\lkj < 0$. The following guarantee of descent is supplied:
\begin{lemma}[Direction of Descent]\label{lem:direction:of:descent}
    Given a feasible point $\xkj$ of~\eqref{eq:LCQP:pen:outer} and inner loop iterate $\pkj = \xkj^\ast - \xkj$, the merit function at $\xkj$ is nonincreasing towards $\pkj$, i.e.
    \begin{equation}\label{eq:descent}
        \nabla \psi(\xkj, \rho_k)^\top (\xkj^\ast - \xkj) \leq 0.
    \end{equation}

    Further, if $\xkj$ is not a stationary point of~\eqref{eq:LCQP:pen:outer} with respect to $\rho_k$, then
    \begin{equation}\label{eq:strict:descent}
        \nabla \psi(\xkj, \rho_k)^\top (\xkj^\ast - \xkj) < 0.
    \end{equation}
\end{lemma}
\begin{proof}
    Since $\xkj^\ast$ is the global minimum of the inner loop optimization problem, the following relation holds
    \begin{equation}
        \thetakj(\xkj^\ast) \leq \thetakj(x),
    \end{equation}
    where $x$ is any feasible point of~\eqref{eq:LCQP:pen:outer}. Since  $\thetakj$ is convex and differentiable it holds for any $a, b \in \R^n$ that
    \begin{equation}
        \nabla \thetakj(a)^\top (b-a) \leq \thetakj(b) - \thetakj(a).
    \end{equation}
    This property provides descent for the quadratic model
    \begin{equation}\label{eq:quadratic:model:descent}
        \nabla \thetakj(\xkj)^\top \pkj \leq \thetakj(\xkj^\ast) - \thetakj(\xkj) \leq 0,
    \end{equation}
    Note that this inequality is strict if $\xkj \neq \xkj^\ast$. Recall that $C \xkj = \dkj$, then the equation
    \begin{subequations}\label{eq:model:and:merit:agree}
        \begin{align}
            \nabla \psi(\xkj, \rho_k)^\top \pkj &= (Q \xkj + \rho_k C \xkj + g)^\top \pkj \\
            &= \nabla \thetakj(\xkj)^\top \pkj
        \end{align}
    \end{subequations}
    shows that the directional derivatives of the merit function and quadratic model at $\xkj$ towards $\pkj$ agree. Inequality~\eqref{eq:descent} follows immediately.

    Assume that $\xkj$ is not outer loop stationary. Then \lemref{lem:inner:outer:stationarity} yields $\xkj^\ast \neq \xkj$. As captured before, the inequality~\eqref{eq:quadratic:model:descent} becomes strict, and again Equation~\eqref{eq:model:and:merit:agree}  shows the statement.
\end{proof}

\begin{figure}
    \centering
    \resizebox{0.95\linewidth}{!}{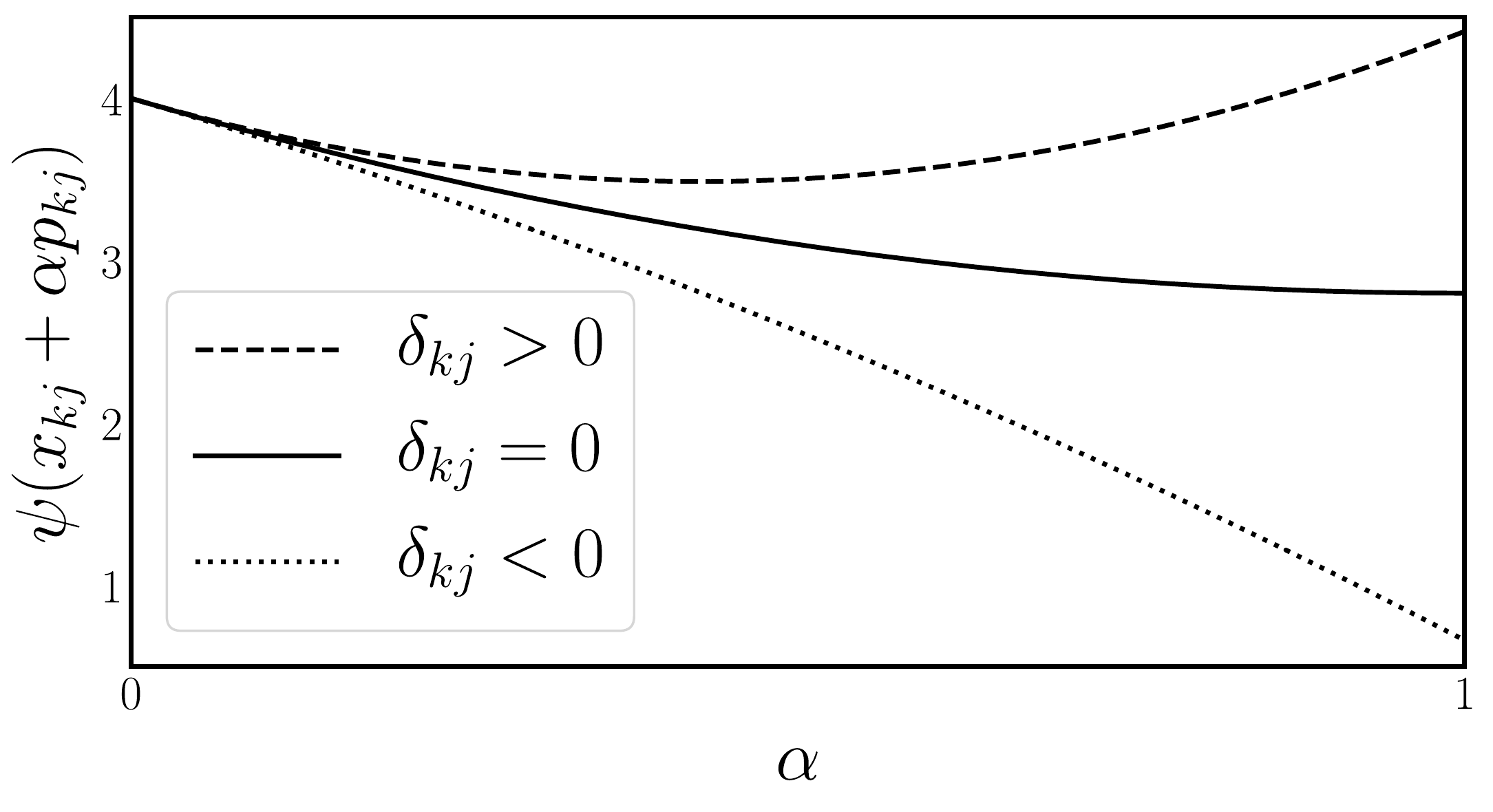}
    \caption{Illustration of the merit function~\eqref{eq:line:search:objective} for $\deltakj > 0$ (dashed), $\deltakj < 0$ (dotted), and the quadratic model $\deltakj = 0$ (solid) along the path $\xkj + \alpha \pkj$ for $\alpha \in [0,1]$. Note that all functions must have equal descent at $\xkj$, as the linearization of the merit function and quadratic model agree at this point (dotted line). Further, note that the quadratic model is minimized for $\alpha = 1$, as it coincides with the unique minimizer of the inner loop subproblem.}
    \label{fig:merit:samples}
\end{figure}

Now consider the sign of $\deltakj$, which determines whether the curvature of the merit function along $\pkj$ is more or less positive than the curvature of the quadratic model. The positive case $\deltakj > 0$ indicates that the merit function along $\pkj$ has a stronger positive curvature, and thus is minimized before the full step. The optimal step length in this case is given by
\begin{equation}\label{eq:step:length:formula}
    \alphakj = \frac{-\lkj}{\qkj}.
\end{equation}
Note that $\qkj > \deltakj > 0$. Assuming that $\xkj$ is not already a stationary point, a strictly positive step length is obtained. The nonpositive case $\deltakj \leq 0$ leads to $\alphakj = 1$, as the curvature of the merit function is less positive than the quadratic model. Both scenarios are visualized in \figref{fig:merit:samples}. Though this strategy provides the best available step length at each iterate, there does not yet exist a statement for sufficient decrease, which would ensure convergence of the inner loop. However, in practice this has not caused any complications.

\subsection{Pseudo Code}\label{subsec:pseudo}
This section provides a pseudo code description of the presented approach, which summarizes all features described before. For simplicity, the inner loop index $j$ is dropped. The algorithm requires the input of an initial guess $\texttt{x0} \in \R^n$ and initial penalty parameter $\texttt{rho} > 0$, as well as a penalty update factor $\texttt{beta} > 1$. The output of this pseudo code is a stationary point of the LCQP~\eqref{eq:LCQP}. The objective's Hessian is denoted by \texttt{Q} and its linear term by \texttt{g}. The matrix \texttt{A} contains all linear constraints, i.e. it includes the rows of $A, L$ and $R$. During the initialization of the QP solver, the {\KKT} matrix is assumed to be factorized. \lineref{line:while:outer} begins the outer loop, where termination is checked and the penalty value is updated. The inner loop begins in \lineref{line:while:inner}, in which the subproblem~\eqref{eq:LCQP:pen:inner} is solved and step updates are performed. The variable \texttt{xn} represents the inner loop minimizer $\xkj^\ast$. The inner loop is terminated if the outer loop stationarity, denoted by \texttt{Stat(xk, yk, rho)}, drops below a certain threshold \texttt{tol}. Similarly, the outer loop is terminated once the complementarity value, denoted by \texttt{Phi(xk)}, falls below this tolerance.
\begin{algorithm}
    \vspace{0.5em}
    \KwIn{\texttt{x0, rho, beta}}
    \KwOut{Stationary point \texttt{(xk, yk)} of LCQP~\eqref{eq:LCQP}}\vspace{0.5em}

    \textcolor{gray}{\# Create QP solver and factorize KKT matrix} \\
    \texttt{qp(Q, g, A, lb, ub, lbA, ubA)}; \label{line:factorize}\\[0.5em]

    \textcolor{gray}{\# Initialize solver with zero penalty QP} \\
    \texttt{(xk, yk) = qp.solve()}; \label{line:init}\\[0.5em]

    \textcolor{gray}{\# Outer loop (penalty update loop)} \\
    \While{\textup{true}}{ \label{line:while:outer} \vspace{0.5em}

        \textcolor{gray}{\# Inner loop (approximate penalty function)} \\
        \While{\textup{\texttt{Stat(xk, yk, rho) > tol}}}{ \label{line:while:inner} \vspace{0.5em}
            \textcolor{gray}{\# Update objective's linear component}\\
            \texttt{qp.update\_g(g + rho*dk)}; \label{line:update:g} \\[0.5em]

            \textcolor{gray}{\# Step computation (solve~\eqref{eq:LCQP:pen:inner})}\\
            \texttt{(xn, yk) = qp.solve()}; \label{line:hotstart} \\[0.5em]

            \textcolor{gray}{\# Get optimal step length according to~\eqref{subsec:globalization}}\\
            \texttt{alpha = StepLength(xk, xn, rho)};\label{line:optimalStepLength} \\[0.5em]

            \textcolor{gray}{\# Perform step}\\
            \texttt{xk = xk + alpha*(xn - xk)}\;
        }\vspace{0.5em}

        \textcolor{gray}{\# Terminate or increase penalty parameter}\\
        \If{ \textup{\texttt{Phi(xk) < tol}}}{ \label{line:comp:check}
        \texttt{return (xk, yk)}\;
        }
        \texttt{rho = beta*rho}\;  \label{line:pen:update}
    }
    \caption{Pseudo code of the approach}
    \label{algo:LCQP}
\end{algorithm}

\section{A Numerical Example}\label{sec:numerical:example}
This section briefly discusses a numerical benchmark by solving the illustrative OCP~\cite[Section 2]{stewart2010optimal}
\begin{mini!}
	{x_0 \in \R, x(\cdot)}{\int_0^2 x(t)^2 \text{d}t + (x(2) - 5/3)^2}
	{\label{eq:IVOCP}}{}
    \addConstraint{x(0)}{= x_0,}{\label{eq:IVOCP:initial:value}}
    \addConstraint{\dot{x}(t)}{\in 2 - \sign (x(t)),}{\quad t \in [0, 2] \label{eq:IVOCP:dynamics}}.
\end{mini!}
The effective degree of freedom in this optimization problem is the initial value $x_0$. Though the constraint~\eqref{eq:IVOCP:dynamics} is a discontinuous ODE, it has a unique solution given by a piecewise linear function with slope $3$ for $x(t) < 0$ and slope $1$ for $x(t) > 0$. The ODE describes a Filippov differential inclusion, which can be reformulated into a dynamic complementarity system~\cite{nurkanovic2020limits}. This method introduces three algebraic variables $y(\cdot), \lambda^-(\cdot), \lambda^+(\cdot)$, which describe the switch in the ODE, the negative part of $x$ and the positive part of $x$, respectively. However, $\lambda^+ = x + \lambda^-$ can be eliminated. This OCP is discretized using implicit Euler in order to obtain the LCQP
\begin{mini*}
    {\substack{x_0,\dots,x_N \in \R \\ y_0,\dots,y_{N-1} \in \R \\ \lambda^-_0,\dots,\lambda^-_{N-1} \in \R}}{\sum_{k=0}^{N-1} E_k(x_k) + E_N(x_N)}
    {\label{eq:IVOCP:LCQP}}{}
    \addConstraint{x_{k-1} + h \bigl(3 - 2y_k\bigr) =}{~x_k,}{\quad 1 \leq k < N,}
    \addConstraint{0 \leq x_k + \lambda^-_k \perp 1 - y_k \geq }{~0,}{\quad 1 \leq k < N,}
    \addConstraint{0 \leq \lambda^-_k \perp y_k \geq }{~0,}{\quad 1 \leq k < N,}
\end{mini*}
where $E_k$ for $0 \leq k < N$ represents the quadrature formula of the implicit Euler discretization, and $E_N$ the terminal cost. The discretized problem represents an LCQP (after a small regularization on the algebraic variables).

In the following, five different solution variants are compared, two of which are based on a MATLAB implementation of the presented algorithm. These two methods differ only in the used linear solver within the QP subproblem solver qpOASES~\cite{Ferreau2014}: one uses the default dense solver, whereas the other utilizes the Schur complement method for which the sparse solver MA57~\cite{ma57} is required. These methods are denoted by LCQP and LCQP Schur, respectively. The remaining three methods are all solved with IPOPT~\cite{wachter2006implementation} through the CasADi interface~\cite{Andersson2018}: one method, denoted by IPOPT Pen, solves the exact same outer loop problems as the LCQP methods, and the other two strategies solve a homotopy of \mbox{(in)equality} regularization schemes. The regularization schemes replace the complementarity product with $x^\top L^\top R x \leq \sigma$ for the relaxed method, and $x^\top L^\top R x = \sigma$ for the smoothed method, both for some $\sigma > 0$ (see \cite{nurkanovic2020limits} for details). These methods are denoted by IPOPT Smoothed and IPOPT Relaxed. The source code of this benchmark is available at \url{https://github.com/hallfjonas/IVOCP}.

\tabref{tab:IVOCP:Values} provides the average complementarity satisfaction together with the average absolute distance to the analytical solution, showing that the proposed algorithm has the highest quality in both aspects. In fact, complementarity is satisfied up to machine precision, which is favoured by having an active-set solver on the subproblem level. Solutions computed with the IPOPT penalty method achieve significantly less precision (due to its conflicting barrier penalty), and the regularization schemes only achieve a low complementarity satisfaction naturally. \figref{fig:IVOCP:TimePlot} shows the average CPU times of this experiment. The introduced method outperforms all other approaches in the first few discretizations. This originates from the fact that the factorization of the KKT matrix is reused, whereas IPOPT is required to recompute it after each penalty update. As the experiments gain size, IPOPT performs better in terms of CPU time than the introduced algorithm, which is due to its exploitation of sparsity structures. However, if a solver like MA57 is available, the LCQP Schur method is able to outperform IPOPT in all experiments. Both regularized methods are unable to compete against the penalty approaches for moderately sized formulations.
\begin{figure}
    \centering
    \resizebox{0.99\linewidth}{!}{
%
%
\definecolor{mycolor1}{rgb}{0.90000,0.80000,0.00000}%
\begin{tikzpicture}

\begin{axis}[%
width=4.602in,
height=2.5in,
at={(0.772in,0.481in)},
scale only axis,
xmin=50,
xmax=150,
xlabel style={font=\color{white!15!black}},
xlabel={Number of discretization nodes},
ymode=log,
ymin=0.04055967,
ymax=2.8,
yminorticks=true,
ylabel style={font=\color{white!15!black}},
ylabel={Average CPU time per OCP $[\mathrm{s}]$},
axis background/.style={fill=white},
xmajorgrids,
ymajorgrids,
yminorgrids,
legend style={at={(0.97,0.03)}, anchor=south east, legend cell align=left, align=left, draw=white!15!black}
]
\addplot [color=black, dashed, line width=1.3pt, mark=square, mark options={solid, black}]
  table[row sep=crcr]{%
50	0.06898567\\
55	0.09245426\\
60	0.11773712\\
65	0.14647106\\
70	0.1870676\\
75	0.23057482\\
80	0.27066695\\
85	0.33157401\\
90	0.40738505\\
95	0.4928816\\
100	0.61719134\\
105	0.7322376\\
110	0.83406295\\
115	0.99433188\\
120	1.16863204\\
125	1.38803694\\
130	1.59405401\\
135	1.75190561\\
140	2.02785308\\
145	2.26604749\\
150	2.59156045\\
};
\addlegendentry{LCQP}

\addplot [color=green!60!darkgray, line width=1.3pt, mark=diamond, mark options={solid, green!60!darkgray}]
  table[row sep=crcr]{%
50	0.04055967\\
55	0.05221558\\
60	0.05916015\\
65	0.06730944\\
70	0.07762515\\
75	0.08969859\\
80	0.10569368\\
85	0.10990305\\
90	0.12651782\\
95	0.1397995\\
100	0.16452122\\
105	0.16376489\\
110	0.17991027\\
115	0.18924835\\
120	0.20447138\\
125	0.22755326\\
130	0.24329334\\
135	0.2554187\\
140	0.27657312\\
145	0.29500807\\
150	0.31114838\\
};
\addlegendentry{LCQP Schur}

\addplot [color=blue, dotted, line width=1.3pt, mark=x, mark options={solid, blue}]
  table[row sep=crcr]{%
50	0.18252535\\
55	0.21191938\\
60	0.23965263\\
65	0.22468275\\
70	0.26034204\\
75	0.29326266\\
80	0.28125606\\
85	0.28210311\\
90	0.31123603\\
95	0.30972475\\
100	0.3703952\\
105	0.42526814\\
110	0.410676\\
115	0.47146203\\
120	0.47705895\\
125	0.50288598\\
130	0.45966498\\
135	0.60762244\\
140	0.53666123\\
145	0.54449966\\
150	0.58659037\\
};
\addlegendentry{IPOPT Pen}

\addplot [color=mycolor1, dashdotted, line width=1.3pt, mark=triangle, mark options={solid, rotate=270, mycolor1}]
  table[row sep=crcr]{%
50	0.25161507\\
55	0.45966039\\
60	0.40839856\\
65	0.31297214\\
70	0.30691685\\
75	0.43094233\\
80	0.42126029\\
85	0.64195437\\
90	0.44871791\\
95	0.56810989\\
100	0.42588849\\
105	0.66409007\\
110	0.71304798\\
115	0.68886524\\
120	0.74511077\\
125	1.2903421\\
130	0.80773393\\
135	0.80889952\\
140	0.73055111\\
145	0.90457585\\
150	0.9295715\\
};
\addlegendentry{IPOPT Smoothed}

\addplot [color=red, dashdotted, line width=1.3pt, mark=o, mark options={solid, red}]
  table[row sep=crcr]{%
50	0.44601697\\
55	0.35966932\\
60	0.43041438\\
65	0.44528994\\
70	0.44470065\\
75	0.50862955\\
80	0.51657739\\
85	0.56513487\\
90	0.62268577\\
95	0.63155478\\
100	0.76542231\\
105	1.04550892\\
110	1.07073167\\
115	1.00530429\\
120	0.90423264\\
125	1.09505399\\
130	1.37484965\\
135	1.38517474\\
140	1.09609889\\
145	1.15655523\\
150	1.33517453\\
};
\addlegendentry{IPOPT Relaxed}

\end{axis}
\end{tikzpicture}%
}
    \caption{Plotting the average CPU time required for each method and discretization size to solve $100$ differently initialized LCQPs.}
    \label{fig:IVOCP:TimePlot}
\end{figure}
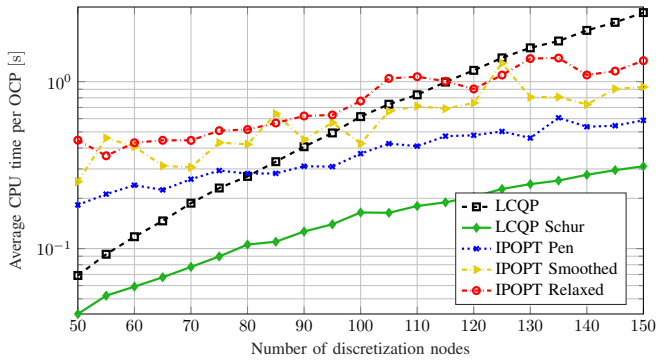

\section{Conclusions and Future Work}\label{sec:conclusion:future:work}
This work presented a novel SCP approach to solving LCQPs. A computationally cheap globalization strategy with the guarantee of merit function descent at each iterate was introduced. Its applicability and promising performance was demonstrated by solving an initial value optimal control problem. A comparison against state-of-the-art solution variants solved by a high performance NLP-solver showed that the algorithm is able to compete in all of the three categories: solution CPU time, complementarity satisfaction, and solution quality. Future work aims at providing an open-source software package to reliably solve LCQPs. The presented algorithm will be implemented with multiple QP solvers on the subsolver level (e.g. qpOASES~\cite{Ferreau2014} and OSQP~\cite{osqp}). Further, the option to solely solve the outer loop homotopy with an adequate solver could be provided. On the theoretical side, future work consists of providing a proof of global convergence regarding both inner and outer loops. Additionally, the presented algorithm could be utilized on a subsolver level for solving nonlinear MPCCs opening up applicability to a wider range of problems.

\bibliographystyle{ieeetr}
\bibliography{bibtex/biblio}

\begin{table}[t]
    \centering
    \ra{1.3}
    \caption{Average values over all experiments.}
    \begin{tabular}{@{}lrr@{}}\toprule
                          & complementarity              & distance to analytical solution \\ \midrule
      LCQP                & $\bm{\expnumber{6.8}{-17}}$  & $\bm{0.018}$      \\
      LCQP Schur          & $\expnumber{2.3}{-16}$       & $\bm{0.018}$      \\
      IPOPT Penalty       & $\expnumber{1.4}{-04}$       & $0.072$ \\
      IPOPT Smoothed      & $\expnumber{1.6}{+04}$       & $0.078$      \\
      IPOPT Relaxed       & $\expnumber{6.0}{+03}$       & $0.61$       \\\bottomrule
    \end{tabular}
    \label{tab:IVOCP:Values}
\end{table}

\end{document}